\documentclass[12pt]{article}

\usepackage{datetime}

\usepackage[all]{xy}

\usepackage{amssymb, amsmath, amsthm,eucal}

\usepackage{mathptmx}
\usepackage[scaled=.90]{helvet}
\usepackage{courier}

\usepackage{graphicx}

\DeclareRobustCommand{\coprod}{\mathop{\text{\fakecoprod}}}
\newcommand{\fakecoprod}{%
  \sbox0{$\prod$}%
  \smash{\raisebox{\dimexpr.9625\depth-\dp0}{\scalebox{1}[-1]{$\prod$}}}%
  \vphantom{$\prod$}%
}

\usepackage[pdftex]{hyperref}

\hypersetup{
  colorlinks = true,
  linkcolor = black,
  urlcolor = blue,
  citecolor = black,
}

\newtheorem{theorem}{Theorem}[section]
\newtheorem{theorem*}{Theorem}
\newtheorem{proposition}[theorem]{Proposition}
\newtheorem{corollary}[theorem]{Corollary}

\theoremstyle{definition}

\renewcommand{\phi}{\varphi}

\renewcommand{\H}{\mathrm{H}}

\newcommand{\e}{\mathrm{e}}

\newcommand{\CC}{\mathbb{C}}
\newcommand{\DD}{\mathbb{D}}
\newcommand{\NN}{\mathbb{N}}
\newcommand{\PP}{\mathbb{P}}
\newcommand{\RR}{\mathbb{R}}
\newcommand{\TT}{\mathbb{T}}
\newcommand{\ZZ}{\mathbb{Z}}

\newcommand{\C}{\mathcal{C}}
\newcommand{\D}{\mathcal{D}}
\newcommand{\M}{\mathcal{M}}

\newcommand{\Mor}{\operatorname{Mor}}
\newcommand{\Emb}{\operatorname{Emb}}

\newcommand{\colim}{\operatorname{colim}}

\newcommand{\B}{\operatorname{B}\!}
\newcommand{\U}{\operatorname{U}}

\newcommand{\Diff}{\operatorname{Diff}}
\newcommand{\BDiff}{\B\Diff}

\newcommand{\Spiffc}{\operatorname{Spiff}^c}
\newcommand{\BSpiffc}{\B\Spiffc}

\title{\bf Stable diffeomorphism groups\\ of~4-manifolds}

\author{Markus Szymik}

\date{June 2008}

\begin{document}

\maketitle

\begin{abstract}%
\noindent
  A localisation of the category of~$n$-manifolds is introduced by
  formally inverting the connected sum construction with a
  chosen~$n$-manifold~$Y$.  On the level of automorphism groups, this
  leads to the stable diffeomorphism groups of~$n$-manifolds. In
  dimensions~0 and~2, this is connected to the stable homotopy groups
  of spheres and the stable mapping class groups of Riemann
  surfaces. In dimension~4 there are many essentially different
  candidates for the~$n$-manifold~$Y$ to choose from. It is shown that
  the Bauer--Furuta invariants provide invariants in the
  case~$Y=\overline{\CC P}^2$, which is related to the birational
  classification of complex surfaces. This will be the case for
  other~$Y$ only after localisation of the target category. In this
  context, it is shown that the~$K3$-stable Bauer--Furuta invariants
  determine the~$S^2\!\times\!S^2$-stable invariants.
\end{abstract}

\thispagestyle{empty}


\section*{Introduction}

One of the main objects of geometric topology is to
classify~$n$-dimensional manifolds up to diffeomorphism -- and their
diffeomorphisms. However, the study of diffeomorphism groups has
turned out to be difficult so far, and one is tempted to address
simpler but related problems first. This has led, for example, to the
introduction of the block diffeomorphism groups, which are relatively
accessible via surgery theory, see \cite{WeissWilliams:Survey} for a
survey. In this work, the focus will be on a different way of
simplifying things: localisation -- in the sense of inverting
operations which need not be invertible originally.  In contrast to
arithmetic localisation techniques, which long have their place in
geometric topology, see~\cite{Sullivan}, we will consider inverting
the operation of connected summation with a chosen~$n$-dimensional
manifold~$Y$.

In the first four sections, after describing the general setup, I will
discuss examples in low dimensions~\hbox{$n=0$} and~\hbox{$n=2$},
which show that this has led to interesting mathematics already,
related to the stable homotopy groups of spheres and the stable
mapping class groups of Riemann surfaces.

The final four sections concentrate on stabilisation of~4-manifolds,
with particular emphasis on the question of how the Bauer--Furuta
invariants relate to this. An essential feature in this dimension is
that there are several reasonable choices of ``directions''~$Y$ in
which to stabilise. Section~\ref{sec:directions} discusses some of
them. The cases~$Y=\overline{\CC P^2}$,~\hbox{$Y=S^2\!\times\!S^2$},
and~\hbox{$Y=K3$} will play a r\^ole later on.

It will turn out that the Bauer--Furuta invariants allow for the
definition of {\it stable} characteristic classes for families
of~4-manifolds. The reader may find it useful to keep the following
analogy with the corresponding approach for vector bundles in mind:
The universal Chern classes live in~$\H^*(\B\U(n))$, where~$\B\U(n)$
is the classifying space for~$n$-dimensional complex vector bundles,
and the fact that Chern classes do not change when a trivial bundle is
added means that there are stable universal Chern classes living
in~$\H^*(\B\U(\infty))$, where~$\B\U(\infty)$ is the classifying space
for stable complex vector bundles. However, caution should be taken
to not to confuse the approach in this article with the stabilisation
process by taking products of manifolds with~$\RR$ as in~\cite{Mazur},
which is closer in spirit to the Chern class picture, but changes the
dimension of the manifolds involved; the connected sum construction
preserves dimension while possibly changing the rank of the
intersection form.

The stable Bauer--Furuta invariants are worked out in
Section~\ref{sec:families} in the case of stabilisation with respect to~\hbox{$Y=\overline{\CC P^2}$},
which is simplified by the fact that the Bauer--Furuta invariant
of~$\overline{\CC P^2}$ is the identity, leading to the existence of a
stable universal characteristic class in the group
\begin{displaymath}
  \pi^0_\TT(\BSpiffc(X\#\infty\overline{\CC
    P}^2,y)^{\lambda(\sigma_X)}),    
\end{displaymath}
see Theorem~\ref{thm:CP2invariants}, where also the issue of
uniqueness is discussed. For general~$Y$, one will have to localise
the target category of the invariants as well. This done in
Section~\ref{sec:localisation}, see Theorem~\ref{thm:stableinvariants}
there, while the final section discusses the
examples~\hbox{$Y=S^2\!\times\!S^2$} and~$Y=K3$. These two cases turn
out to be related in the sense that the~$K3$-stable invariants
determine the~\hbox{$S^2\!\times\!S^2$}-stable invariants, see
Theorem~\ref{thm:K3determinesTorus}. This would be a trivial result
if~\hbox{$S^2\!\times\!S^2$} were a connected summand of~$K3$; but it
is not.


\section{Based manifolds}

The aim of this section is to give a description of a category of
manifolds where a common construction, connected summation with
another manifold, is well-defined -- not up to diffeomorphism, but on
the nose. There will be reasons to restrict the topology of the
manifolds involved. For example, one might only want to consider
connected or simply-connected manifolds. And there will be reasons to
consider manifolds with orientations, even in the generalised sense of
spin structures or framings. However, it should be clear how to adapt
the following to the specified contexts if necessary.

We will consider closed manifolds of a fixed
dimension~$d$. Let~$\M_{d}$ be the category of pairs~$(X,x)$,
where~$X$ is closed~$d$-manifold and
\begin{displaymath}
  x\colon\DD^{d}\longrightarrow X 
\end{displaymath}
is an embedding of the closed~$d$-disk into~$X$. The space of
embeddings~$\DD^{d}\rightarrow X$ is homotopy equivalent to the frame
bundle of~$X$. Thus, these embeddings~$x$ can be thought of as
manifolds with a framed base point~$x(0)$. However, it will be
important to have an actual embedding as part of the structure. The
morphisms in~$\M_{d}$ from~$(X,x)$ to~$(X',x')$ are the diffeomorphisms
from~$X$ to~$X'$ which send~$x$ to~$x'$ in the sense that the triangle
\begin{center}
  \mbox{ 
    \xymatrix{
      & \DD^d\ar[dl]_x\ar[dr]^{x'} & \\
      X\ar[rr]_\cong && X'
    } 
  }
\end{center}
commutes. By construction, the category~$\M_{d}$ is a groupoid. This
will hold for all categories considered here. The class of objects
could be considered with the topology from the embedding spaces, but
this will not be done here. However, the automorphism group of~$(X,x)$
in~$\M_{d}$ can also be considered with its natural topology, and this
will be done here. That automorphism group is the group~$\Diff(X,x)$
of all diffeomorphisms of~$X$ which fix~$x$.

Let us compare the groupoid~$\M_{d}$ of based manifolds with the
groupoid of all closed~$d$-manifolds and diffeomorphisms. If two
objects~$(X,x)$ and~$(X',x')$ are isomorphic in~$\M_{d}$, then~$X$
and~$X'$ are diffeomorphic. The converse holds if~$X$~(and therefore
also~$X'$) is connected. 

If~$X$ is connected, this automorphism group in~$\M_{d}$ is only a
frame bundle away from the group~$\Diff(X)$ itself, in the sense that
there is a fibration sequence
\begin{displaymath}
  \Diff(X,x)\longrightarrow\Diff(X)\longrightarrow\Emb(\DD^d,X),
\end{displaymath}
and that~$\Emb(\DD^d,X)$ is homotopy equivalent to the frame bundle
of~$X$. For example, take~$X=S^d$. Then the composition
\begin{displaymath}
  SO(d+1)\longrightarrow\Diff(S^d)\longrightarrow\Emb(\DD^d,S^d)
\end{displaymath}
is an equivalence. In fact, for~$d=1,2,3$ both arrows are
equivalences, and the group~$\Diff(S^2,x)$ is contractible, see~\cite{Smale}
and~\cite{Hatcher}. This is false for~\hbox{$d\geqslant
  5$}~\cite{Milnor}, and unknown (at present) for~\hbox{$d=4$}.

To sum up, the category~$\M_{d}$ of based manifolds is sufficiently
close to the category of unbased manifolds that their difference is
under control. It is time to see what the base is good for.


\section{Connected sums}

Let us fix a closed~$d$-manifold~$Y$ and an embedding
\begin{displaymath}
  (y',y)\colon \DD^{d}+\DD^{d}\longrightarrow Y.  
\end{displaymath}
If~$Y$ is connected, the choice of~$y$ and~$y'$ will not
matter and will be omitted from the notation. In any case,
this yields an endofunctor~$F_{Y}$ of~$\M_{d}$ which is given
on objects by connected summation:
\begin{displaymath}
  F_{Y}(X,x)=(X\# Y, y),
\end{displaymath}
where the connected sum~$X\# Y$ is constructed from~$X\backslash x(0)$
and~$Y\backslash y'(0)$ by the usual identifications. (It is here,
where the actual embedding is used.) On morphisms, the functor~$F_{Y}$
acts by extending a diffeomorphism of~$X$ which fixes~$x$ over~$X\# Y$
by the identity on~$Y$.

The question arises whether this functor is invertible or not. Of
course, it will rarely be invertible in the strict sense: here and in
the following, functors will be considered only up to natural
isomorphism.

\begin{proposition}\label{prop:sphere}
  The functor~$F_{Y}$ is invertible if and only if~$Y$ is a
  homotopy sphere.
\end{proposition}

\begin{proof}
  Let~$Y$ be an homotopy sphere. Then there is another
  homotopy sphere~$Z$ such that~$Y\#Z\cong S^d$. Connected
  sum with~$Z$ gives the inverse.
  
  Let~$F_Y$ be invertible. Then there is a manifold~$X$ such
  that~$X\#Y\cong S^d$. It follows that~$X$ and~$Y$ are homotopy
  spheres. See~\cite{Milnor:Sommes}.
\end{proof}

In any case, if the functor~$F_{Y}$ is not invertible, it
can be formally inverted, and this it what will be done
next.


\section{Formally inverting endofunctors}\label{sec:endofunctors}

Given a category~$\C$ with an endofunctor~$F$, there is a universal
category~$\C[F^{-1}]$ with an autofunctor, denoted by~$\overline F$,
and a functor~\hbox{$\C\rightarrow\C[F^{-1}]$} compatible with the
functors~$F$ and~$\overline F$, which is universal (initial) among
such functors: if~$(\D,G)$ is another category with an autofunctor,
and if~$\Phi\colon(\C,F)\rightarrow(\D,G)$ is a functor compatible
with~$F$ and~$G$, then there is a unique functor
$\phi\colon(\C[F^{-1}],\overline F)\rightarrow(\D,G)$ such that the
diagram
\begin{center}
  \mbox{ \xymatrix{ (\C,F)\ar[d]_\Phi\ar[r] &
      (\C[F^{-1}],\overline F)\ar@{-->}[dl]^\phi \\
      (\D,G) } }
\end{center}
commutes.

There are several (naturally equivalent) constructions of $\C[F^{-1}]$
available. In one of them, the objects are the pairs~$(C,n)$,
where~$C$ is an object of~$\C$ and~$n$ is an integer. The set of
morphisms from~$(C,n)$ to~$(C',n')$ is
\begin{equation}\label{eq:defmor}
  \underset{m}{\colim}\Mor_{\C}(F^{m+n}C,F^{m+n'}C'),
\end{equation}
where the colimit is formed using the maps induced
by~$F$. For example, the identity of~$C$ represents a natural
isomorphism
\begin{equation}\label{eq:iso}
  (C,1)\cong(FC,0)
\end{equation}
in~$\C[F^{-1}]$. The functor from~$\C$ to~$\C[F^{-1}]$ sends~$C$
to~$(C,0)$, and the functor~$F$ on~$\C$ extends to~$\C[F^{-1}]$ by
acting on the first component. An isomorphism like~(\ref{eq:iso}) shows
that this extension of~$F$ is naturally isomorphic to the functor
$\overline F$ on~$\C[F^{-1}]$ which sends~$(C,n)$ to~$(C,n+1)$. As the
latter is clearly invertible, so is the former. As for the universal
property,~$\phi$ must be defined on objects by~$\phi(C,n)=G^n\Phi
C$. If~$c\colon F^{m+n}C\rightarrow F^{m+n'}C'$ represents a morphism
$[c]\colon(C,n)\rightarrow(C',n')$, then~$\phi[c]$ is to be defined
such that~$G^m\phi[f]=\Phi(f)$.  See~\cite{Margolis} for all this in a
similar context.

A different model for the category~$\C[F^{-1}]$ is a special case of Quillen's
construction, see~\cite{Grayson:QuillenII}, namely the
category~$\NN^{-1}\mathcal C$, where the monoid~$\NN$ is interpreted
as a (discrete) monoidal category, acting on~$\mathcal C$ via~$F$. As
it turns out, this model is literally the same as the Grothendieck
construction in the case of the
diagram~\hbox{$\mathcal{C}\overset{F}{\rightarrow}
  \mathcal{C}\overset{F}{\rightarrow}
  \mathcal{C}\overset{F}{\rightarrow}\dots$} of categories.  As
Thomason proved, in~\cite{Thomason}, there is an equivalence
\begin{displaymath}
  \B\mathcal C[F^{-1}]
  \simeq
  \operatorname{hocolim}(
  \B\mathcal{C}\overset{\B F}{\longrightarrow}
  \B\mathcal{C}\overset{\B F}{\longrightarrow}
  \B\mathcal{C}\overset{\B F}{\longrightarrow}
  \dots), 
\end{displaymath}
which also implies that this construction is well-behaved with respect
to homology.

Two objects~$C$ and~$C'$ of~$\mathcal C$ are called {\it~$F$-stably
  isomorphic} if their images in~$\mathcal C[F^{-1}]$ are isomorphic;
this is the case if and only if there is a non-negative integer~$n$
such that~$F^nC$ and~$F^nC'$ are isomorphic in~$\mathcal C$. Two
objects~$C$ and~$C'$ of~$\mathcal C$ are called~{\it$F$-stably
  equivalent} if there are non-negative integers~$n$ and~$n'$ such
that~$F^nC$ and~$F^{n'}C'$ are isomorphic in~$\mathcal C$. Clearly,
two~$F$-stably isomorphic objects are~$F$-stably equivalent, and the
converse need not hold. The importance of the notion of equivalence
lies in the following fact.

\begin{proposition}\label{prop:aut}
  The isomorphism type of the automorphism group of~$C$ in~$\mathcal
  C[F^{-1}]$ depends only on its~$F$-stable equivalence class.
\end{proposition}

\begin{proof}
  It suffices to prove that~$C$ and~$FC$ have isomorphic automorphism
  groups in~$\mathcal C[F^{-1}]$. But that follows immediately from
  the definition~(\ref{eq:defmor}).
\end{proof}


\section{Stable Diffeomorphism Groups}

The abstract construction of the previous section can now be applied
to the category~$\M_{d}$ with the
endofunctor~$F_Y$. Let~$\M_d[Y^{-1}]$ be the category which is
obtained from~$\M_{d}$ by formally inverting~$F_{Y}$.  Given any
object~$(X,x,n)$ in~$\M_d[Y^{-1}]$, its automorphism group is
\begin{equation}\label{eq:colimit}
  \Diff(X\#\infty Y)\;\stackrel{\text{\tiny def}}{=}\;
  \underset{n}{\colim}\Diff(X\#nY),
\end{equation}
where again the embeddings have been suppressed from the
notation. See~(\ref{eq:interpretation}) below for a geometric
interpretation of the maps involved in the colimit on the level of
classifying spaces. The groups~(\ref{eq:colimit}) are the {\it stable
  diffeomorphism groups} of~$X$ with respect to~$Y$ to which the title
refers. As has been pointed out in Proposition~\ref{prop:aut}, up to
isomorphism, the $Y$-stable diffeomorphism group of $X$ depends only
on the $Y$-stable equivalence class of $X$.

In the rest of this section, manifolds of dimension 0 and 2 will be
studied from the point of view of their stable diffeomorphism
groups. The underlying mathematics is well-known, and the only point
of repeating it here is to illustrate the fact that the abstract setup
from the previous section leads to interesting mathematics even in the
simplest cases. The remaining sections will discuss stable
diffeomorphism groups of~4-manifolds.

\subsection{Dimension 0}

The category~$\M_0$ is the category of finite pointed sets and their
pointed bijections. The classifying space is
\begin{displaymath}
  \B\M_0\simeq\coprod_{n\geqslant0}\B\Sigma_{n}.
\end{displaymath}
For stabilisation, one needs~$Y$ to have at least two elements. As
shown in Proposition~\ref{prop:sphere}, the case~$Y=S^0$ is
uninteresting. One could use~$S^{0}\times S^{0}$, but that turns out
to be a connected sums in this case: it is the connected double of the
set~$Y$ with three elements. For any pointed finite set, the connected
sum~$X\# Y$ has one element more than~$X$. There is just
one~$Y$-stable equivalence class of objects, and its automorphism
group is the infinite symmetric group~$\Sigma_\infty$. This gives
\begin{displaymath}
  \B\M_0[Y^{-1}]\simeq\ZZ\times\B\Sigma_{\infty}.
\end{displaymath}
This space has the same homology as the infinite loop space associated
to the sphere spectrum, so that the unstable homotopy groups of its
plus construction are the stable homotopy groups of spheres. Note that
a group completion of~$\B\M_0$ agrees with the plus construction
of~$\B\M_0[Y^{-1}]$, see~\cite{Adams}.

\subsection{Dimension 2}

Let us now consider 2-manifolds which are connected and oriented only,
retaining the notation~$\M_2$ for the corresponding
subcategory. These manifolds are classified by their genus~$g$. As we
consider only diffeomorphisms which fix an embedded disk, the
classifying spaces of the diffeomorphism groups are homotopically
discrete, see~\cite{EarleSchatz} or~\cite{Gramain}. In fact, they are
homotopy equivalent to the classifying spaces of the corresponding
mapping class groups~$\Gamma_{g,1}$. This gives
\begin{displaymath}
  \B\M_2\simeq\coprod_{g\geqslant 0}\B\Gamma_{g,1}.
\end{displaymath}
Stabilisation with respect to the torus~$Y=S^{1}\times
S^{1}$ gives
\begin{displaymath}
  \B\M_2[Y^{-1}]\simeq\ZZ\times\B\Gamma_{\infty,1}.
\end{displaymath}
Also in this case, a group completion of~$\B\M_2$ agrees with the plus
construction of~$\B\M_2[Y^{-1}]$, see~\cite{MadsenWeiss} and the
references within. The monoidal structure on~$\B\M_2$ in question is
the pair-of-pants multiplication, see~\cite{Miller}. It generalises to
higher dimensions, and extends to an action of the little $d$-disks
operad on~$\B\M_d$. This implies that the group completion of~$\B\M_d$
is a $d$-fold loop space. It would be interesting to know whether or
not this is actually an infinite loop space as in the case $d=2$.


\section{Stabilisation in Dimension~4}\label{sec:directions}

In this section, we will consider simply-connected oriented
4-manifolds. There is not even a good conjecture what the set of
isomorphism classes could be, and the entire space~$\B\M_4$ seems to
be far beyond reach at present. In order to simplify this, there are
several different directions~$Y$ in which one could try to
stabilise. The following discusses some choices which are reasonable
from one or another perspective.

\subsection{The case~\texorpdfstring{$Y=\overline{\CC P}^2$}{Y=CP2}} 

Stabilisation with respect to~$Y=\overline{\CC P}^2$ is motivated by
complex algebraic geometry. Two complex algebraic surfaces are
birationally equivalent if and only if they are related by a sequence
of blow-ups. From the point of view of differential topology, the
blow-up of a surface~$X$ is diffeomorphic to~$X\#\overline{\CC
  P}^2$. 

\begin{proposition}
  If two complex algebraic surfaces are birationally equivalent, then
  they are~$\overline{\CC P}^2$-stably equivalent, but (in general)
  not conversely.
\end{proposition}

\begin{proof}
  The first part of the statement is clear from the discussion above.
  
  While diffeomorphic surfaces will {\it a fortiori} be~$\overline{\CC
    P}^2$-stably equivalent, they need not be algebraically
  isomorphic. In fact, minimal models for non-ruled surfaces are
  algebraically unique, see~III~(4.6) in~\cite{BPV} for example. Thus,
  it suffices to find two minimal surfaces which are diffeomorphic but
  not algebraically isomorphic, and there are plenty of those.
\end{proof}

In any case, this discussion leads to the question of finding smooth
minimal models: repre\-sen\-ta\-tives of the~$\overline{\CC P}^2$-stable
equivalence classes. This is richer than the theory of complex
algebraic surfaces:

\begin{proposition}
  Not every~$\overline{\CC P}^2$-stable equivalence class is
  representable by a complex surface.
\end{proposition}

\begin{proof}
  For complex surfaces, the sum of the Euler characteristic and the
  signature is divisible by 4 by Noether's formula. Therefore, the
  sphere~$S^4$~(and the connected sum~$\CC P^2\#\CC P^2$ and...) is
  not~$\overline{\CC P}^2$-stably equivalent to a complex surface.
\end{proof}

\subsection{The case~\texorpdfstring{$Y=S^2\!\times\!S^2$}{Y=S2xS2}} 

This is the classical case, and it corresponds to stabilising the
intersection form with respect to hyperbolic planes. It is known, by a
result of Wall's, see~\cite{Wall:Manifolds}, that two
simply-connected~4-manifolds are~\hbox{$(S^2\!\times\!S^2)$}-stably
diffeomorphic if and only their intersection forms are isomorphic. It
is now easy to make a list of the~\hbox{$(S^2\!\times\!S^2)$}-stable
equivalence classes.

\begin{proposition}
  The different~\hbox{$(S^2\!\times\!S^2)$}-stable equivalence classes
  are
  \begin{displaymath}
    S^4\qquad mK3\qquad m\CC P^2\#\overline{\CC P}^2
    \qquad \CC P^2\#m\overline{\CC P}^2,
  \end{displaymath}
  where~$m\geqslant 1$ and~$mX$ is again short for the~$m$-fold
  connected sum of~$X$ with itself.
\end{proposition}

\begin{proof}
  We have to represent all the stable isomorphism classes of
  intersection forms of 4-manifolds with respect to orthogonal
  summation with hyperbolic planes. In case the form is even, it is
  stably determined by the~(even) number~$2m$ of~$E_8$ summands. These
  are stably represented by~$mK3$~(or~$S^4$ if~$m=0$). In case the
  form is odd, the orthogonal sum with the hyperbolic plane is
  indefinite, so these are stably represented by~$m\CC
  P^2\#n\overline{\CC P}^2$ with some~$m,n\geqslant 1$. But the
  existence of a diffeomorphism
  \begin{equation}\label{eq:diffeo}
    (S^2\!\times\!S^2)\#(\CC P^2\#\overline{\CC
      P}^2)\cong(\CC P^2\#\overline{\CC P}^2)\#(\CC
    P^2\#\overline{\CC P}^2)
  \end{equation}
  implies that, stably, one of~$m$ and~$n$ may be chosen to be~$1$.
\end{proof}

As for the~$(S^2\times S^2)$-stable diffeomorphism groups, only the
groups of components, the stable mapping class groups, have been
studied so far, again initiated by Wall~\cite{Wall:Diff}. See
also~\cite{Quinn}.

The~$(\CC P^2\#\overline{\CC P}^2)$-stable case is related, but
different. Note that the existence of a diffeomorphism~(\ref{eq:diffeo})
implies that~$(\CC P^2\#\overline{\CC P}^2)$-stabilisation is coarser
than stabilisation with respect to~$(S^2\!\times\!S^2)$; it neglects
the distinction between (real) spin and
non-spin~4-manifolds. See~\cite{Giansiracusa} for more on this case.

\subsection{The case~\texorpdfstring{$Y=K3$}{Y=K3}}\label{subsec:K3I}

This is related to the previous example, since the intersection form
of~$S^2\!\times\!S^2$ is an orthogonal summand in that of~$K3$, but
different, since~$S^2\!\times\!S^2$ is not a connected summand
of~$K3$. The final section discusses this at the level of the (stable)
Bauer--Furuta invariants, see Theorem~\ref{thm:stableinvariants}. At
present, it is unknown whether any two homotopy equivalent,
simply-connected~4-manifolds are~$K3$-stably diffeomorphic, see
Problem~4.7 in~\cite{Bauer:Sum}.


\section{Bauer--Furuta invariants}\label{sec:families}

The aim of this section is to show that the Bauer--Furuta invariants,
see~\cite{BauerFuruta}, can be used to define {\it stable}
characteristic classes for families of~4-manifolds.  Let me first
recall how they have been extended, in~\cite{Szymik:Families}, to
define {\it unstable} characteristic classes for families of
4-manifolds. From now on, as already in the previous section, all
4-manifolds will be assumed to be simply-connected.

The Bauer--Furuta invariant of a~4-manifold~$X$ depends on the choice
of a complex spin structure~$\sigma_X$ on~$X$. Complex spin families
are classified by an extension~$\BSpiffc(X,\sigma_X)$ of~$\BDiff(X)$
by the gauge group, which in this case is equivalent to~$\TT$ again.
Let~$\lambda=\lambda(\sigma_X)$ be the virtual vector bundle
over~$\BSpiffc(X,\sigma_X)$ which is the difference of the index
bundle of the Dirac operator (associated to~$\sigma_X$) and the bundle
of self-dual harmonic~2-forms, and denote its Thom spectrum
by~$\BSpiffc(X,\sigma_X)^\lambda$. Then there is a universal
characteristic class for complex spin families with typical
fibre~$(X,\sigma_X)$ living in the 0-th stable cohomotopy
group~$\pi^0_\TT(\BSpiffc(X,\sigma_X)^\lambda)$ of this Thom
spectrum. The family~$X$ over the singleton is classified by a
map~$S^\lambda\rightarrow\BSpiffc(X,\sigma_X)^\lambda$, along which
the universal class pulls back to the Bauer--Furuta invariant
of~$(X,\sigma_X)$ living in~$\pi^0_\TT(S^\lambda)$.

By Bauer's connected sum theorem, see~\cite{Bauer:Sum}, the
Bauer--Furuta invariant of a connected sum~$(X\# Y,\sigma_{X\# Y})$ is
the smash product
\begin{displaymath}
  S^{\lambda(X\# Y)}\cong
  S^{\lambda(X)}\wedge S^{\lambda(Y)}
  \longrightarrow S^0\wedge S^0= S^0
\end{displaymath}
of the invariants of the summands. We will need an extension of that
result to families in order to define stable characteristic classes
below. This will involve based manifolds, so to keep notation
reasonable, the complex spin structure will be omitted from it if
confusion seems unlikely.

Let~$(X,x)$ be a based complex spin~4-manifold, and~$(Y,y',y)$ be a
complex spin~4-manifold with respect to which we want to
stabilise. If~$\mathcal X$ is a based family of complex
spin~4-manifolds with typical fibre~$(X,x)$ over~$B$, the base yields
a thickened section~$B\times \DD^4\rightarrow\mathcal X$ of the
projection~$\mathcal X\rightarrow B$. Such families are classified by
maps from~$B$ to~$\BSpiffc(X,x)$. Similarly, the product family~$Y_B$
over~$B$ with fibre~$Y$ comes with two disjoint thickened sections. The
fibrewise connected sum of~$\mathcal X$ and~$Y_B$ is another based
family over~$B$, say~$\mathcal X\#_BY_B$, this time with typical
fibre~$X\#Y$. If~$\mathcal X$ is the universal family
over~$B=\BSpiffc(X,x)$, the family~$\mathcal X\#_BY_B$ is classified
by a map
\begin{equation}\label{eq:interpretation}
  \BSpiffc(X,x)\longrightarrow\BSpiffc(X\#Y,y).
\end{equation}
This gives a geometric interpretation, on the level of classifying
spaces, of the maps in the colimit~(\ref{eq:colimit})
defining~$\BSpiffc(X\#\infty Y,y)$.

\begin{proposition}\label{prop:glueing}
  With the notation from the paragraph above, the family invariant
  of~$\mathcal X\#_BY_B$ is the fibrewise smash product of the family
  invariants of~$\mathcal X$ and~$Y_B$.
\end{proposition}

\begin{proof}
  As long as the glueing happens over product bundles of
  cylinders, say~\hbox{$B\times (S^3\times I)$}, Bauer's proof of his
  connected sum theorem given in~\S 2 and~\S 3 of~\cite{Bauer:Sum}
  can be adapted to families using the same homotopies extended
  constant in the~$B$-direction. This is exactly what the thickened
  sections have been chosen for in the case at hand.
\end{proof}

The map~(\ref{eq:interpretation}) induces a map
\begin{displaymath}
  \BSpiffc(X,x)^{\lambda(X\#Y)}\longrightarrow\BSpiffc(X\#Y,y)^{\lambda(X\#Y)}
\end{displaymath}
between Thom spectra. The bundle used on the left hand side is the
pullback of the bundle used on the right hand side under the
map~\eqref{eq:interpretation}. As the family~$\mathcal X\#_BY_B$
over~$\BSpiffc(X,x)$ is a fibrewise connected sum, this pullback
decomposes as~\hbox{$\lambda(X)\oplus\lambda(Y)$}. As the family $Y_B$
is trivial, the latter bundle $\lambda(Y)$ is trivial. This leads to
an identification
\begin{displaymath}
  \BSpiffc(X,x)^{\lambda(X\#Y)}\simeq\BSpiffc(X,x)^{\lambda(X)}\wedge
  S^{\lambda(Y)}.
\end{displaymath}
Using this and Proposition~\ref{prop:glueing} above, it follows by
naturality of the family invariants that the induced map in cohomotopy
sends the invariant of the universal family over~$\BSpiffc(X\#Y,y)$ to
the fibrewise smash product of the invariant of the universal family
over~$\BSpiffc(X,x)$ with the invariant of the product family~$Y_B$.

The easiest case would be that the invariant of the product
family~$Y_B$ is the identity over~$B$. This happens
for~$Y=\overline{\CC P}^2$ with its standard complex spin structure.

\begin{proposition}
  The Bauer--Furuta invariant of~$\overline{\CC P}^2$ is the class of
  the identity~$S^0\rightarrow S^0$.
\end{proposition}


This is well-known, see~\cite{Bauer:Sum}. Thus, the map induced by
\begin{displaymath}
  \BSpiffc(X,x)\longrightarrow\BSpiffc(X\#\overline{\CC P}^2,y)
\end{displaymath}
in the cohomotopy of the Thom spectra sends the invariant of the
universal family, which lives over~$\BSpiffc(X\#\overline{\CC
  P}^2,y)$, to the invariant of the universal family
over~$\BSpiffc(X,x)$. These classes therefore define an element in
\begin{displaymath}
	\lim_n\pi^0_\TT(\BSpiffc(X\#n\overline{\CC P}^2,y)^{\lambda(X)}).
\end{displaymath}
By the $\lim$-$\lim^1$-sequence, there exists a universal class over
the colimit of these Thom spectra, and the indeterminacy of that class
is the group
\begin{displaymath}
		\lim_n{\!}^1\pi^{-1}_\TT(\BSpiffc(X\#n\overline{\CC P}^2,y)^{\lambda(X)}).
\end{displaymath}
Given a tower of finite groups $G_n$, the completion theorem
(conjectured by Segal) implies that the groups $\pi^{-1}(BG_n)$ are
finite, so that the Mittag-Leffler condition is satisfied
and~\hbox{$\lim^1=0$} in this toy situation. In general, there seems
to be no reason why this should be the case, and its seems best to
honour the coset we have, which restricts to a unique element on
every finite stage $n$; this is all that is really needed. In this
qualified sense, the following results.

\begin{theorem}\label{thm:CP2invariants}
  The Bauer--Furuta invariants define a stable universal characteristic
  class which lives in the
  group
  \begin{displaymath}
	\pi^0_\TT(\BSpiffc(X\#\infty\overline{\CC P}^2,y)^{\lambda(X)}).
	\end{displaymath}
\end{theorem}

As a corollary of this, which may also be deduced from~\cite{Bauer:Sum}
directly, the Bauer--Furuta invariants of~$X$ and~$X\#\overline{\CC
  P}^2$ agree:

\begin{corollary}
  The Bauer--Furuta invariants only depend on~$\overline{\CC
    P}^2$-stable equivalence classes.
\end{corollary}

The strength of the theorem as compared to its corollary lies in the
fact that it gives information on the entire classifying space, not
just on the set of its components.


\section{Localisation}\label{sec:localisation}

In the previous section, we have seen that the Bauer--Furuta classes
are invariants of the~$\overline{\CC P}^2$-stable diffeomorphism
category of complex spin~4-manifolds. The reason for this was that the
invariant of~$\overline{\CC P}^2$ itself is the identity map. However,
for the general~$Y$-stable case, the invariant of~$Y$ need not be the
identity. In fact, it need not even be invertible; but, it can be
formally inverted as in Section~\ref{sec:endofunctors}. This will be
done in this section, in order to define invariants of the~$Y$-stable
diffeomorphism type of~$X$ even in the case when the invariant of~$Y$
is not invertible. For clarity, the focus will be on plain manifolds
here; general families should be treated using the notation from the
previous section.

If~\hbox{$f\colon A\rightarrow B$} is a stable map, smashing with~$f$
defines the morphism groups
\begin{displaymath}
  [M,N]_f 
  \;\stackrel{\text{\tiny def}}{=}\; 
  \mathrm{colim}_k[M\wedge A^{\wedge k},N\wedge B^{\wedge k}]
\end{displaymath}
in the localisation of the stable homotopy category with respect to
the endofunctor~$?\wedge f$. Similar notation will be used in
the equivariant case. Localisation away from Euler classes of
representations has a long tradition~\cite{tomDieck}.

This construction will now be applied in the case~$f=m(Y)\colon
S^{\lambda(Y)}\rightarrow S^0$, the Bauer--Furuta invariant of the
4-manifold~$Y$ with respect to which we want to stabilise. For every
complex spin 4-manifold~$X$ as before, the Bauer--Furuta invariant of
the connected sum~$X\# Y$ is obtained from that of~$X$ by smashing it
with~$m(Y)$, by Bauer's theorem again. This implies that the sequence of
invariants~$m(X\#kY)$, for varying~$k$, defines an element
\begin{equation}\label{TT-invariant}
  m(X\#\infty Y)\in[S^{\lambda(X)},S^0]^\TT_{m(Y)}
\end{equation}
in the localisation of the~$\TT$-equivariant stable homotopy category
with respect to~$m(Y)$. It turns out, however, that this is
practically useless, except for manifolds~$Y$ with~$b^+(Y)=0$, such as
$Y=\overline{\CC P}^2$, since the~$\TT$-equivariant Bauer--Furuta
invariants are known to be nilpotent
otherwise. See~\cite{FurutaKametaniMinami}, for example. And clearly,
if~$f$ is nilpotent, then the localisation with respect to~$f$ leads
to trivial groups.

However, if~$Y$ is real spin, it is known that there is a lift of the
Bauer--Furuta invariant~$m(Y)$ from the~$\TT$-equivariant to the
$\PP$-equivariant stable homotopy category,
see~\cite{BauerFuruta}. Here and in the following, the
group~$\PP=\mathrm{Pin}(2)$ is the normaliser of~$\TT$
inside~$\mathrm{Sp}(1)$; it sits in an extension
\begin{displaymath}
  1\longrightarrow \TT\longrightarrow \PP\longrightarrow 
  \ZZ/2\longrightarrow1.
\end{displaymath}
One also has to specify a universe for the group~$\PP$, which in this case is of the
form~\hbox{$\RR^\infty\oplus D^\infty\oplus H^\infty$}, where~$\RR$ is
the trivial~$\PP$-line,~$D$ is the line with the action induced by the
antipodal action of~\hbox{$\PP/\TT=\ZZ/2$}, and~$H$ is
the~4-dimensional tautological quaternion action. This universe is
understood from now on. A useful reference for the homotopy theory in
this context is~\cite{Birgit}.

In the situation leading to~(\ref{TT-invariant}), if~$X$ and~$Y$ are
real spin, so is~$X\#kY$ for all~$k$. The same reasoning as above
yields the following result.

\begin{theorem}\label{thm:stableinvariants} Let~$X$ and~$Y$ be real
  spin 4-manifolds. The sequence of invariants~$m(X\#kY)$, for varying~$k$,
  defines an element
  \begin{displaymath}
    m(X\#\infty Y)\in[S^{\lambda(X)},S^0]^\PP_{m(Y)}
  \end{displaymath}
  in the localisation of the~$\PP$-equivariant stable homotopy
  category with respect to~$m(Y)$.
\end{theorem}

The invariant from the previous theorem will be referred to as the
{\it ~$Y$-stable Bauer--Furuta invariant} of~$X$. In the following
section, the general theory will be illustrated for two~$Y$ which are
non-trivial in the sense that their~$\PP$-equivariant Bauer--Furuta
invariants are non-nilpotent,~$Y=S^2\!\times\!S^2$ and~$Y=K3$.

The rest of this section will contain two general remarks, both
related to $\PP$-fixed points. Here and in the following, the notation
$\Phi^\PP$ will be used for the geometric fixed point functor, which
sends the suspension spectrum of a $\PP$-space to the suspension
spectrum of its $\PP$-fixed points. See~\cite{Mayetal}, XVI.3.

First, the~$\PP$-equivariant Bauer--Furuta invariant of a real spin
4-manifold, stable or not, will always restrict to the identity map of
$S^0$ on~$\PP$-fixed points. This is a consequence of the
$\PP$-actions used on the source and target of the monopole map to
make this map $\PP$-equivariant: the group~$\PP$ acts on spinors via
the representation~$H$ and on forms via $D$; the trivial
representation does not occur. Therefore, these invariants are never
nilpotent. As the identity is invertible, there is an induced dashed
arrow in the diagram
\begin{center}
  \mbox{ 
    \xymatrix{
      [M,N]^\PP\ar[r]\ar[rd]_-{\Phi^\PP}& [M,N]^\PP_{m(Y)}\ar@{-->}[d]\\
      & [\Phi^\PP(M),\Phi^\PP(N)],
    } 
  }
\end{center}
and the observation above may be rephrased to say that the image of a
$Y$-stable Bauer--Furuta invariant will always map down to the identity
of $S^0$. This clearly gives restrictions on the possible values of
the invariants.

Second, passage to~$\PP$-fixed points is also a localisation in the
situation at hand.

\begin{proposition} There is a natural isomorphism 
  \begin{displaymath}
    [\Phi^\PP M,\Phi^\PP N]\cong[M,N]^\PP_{\e(D\oplus H)}
  \end{displaymath}
  for all~$\PP$-spectra~$M$ and~$N$ indexed on our~$\PP$-universe.
\end{proposition}

\begin{proof}
  In general, there is an isomorphism
  \begin{displaymath}
    [\Phi^\PP M,\Phi^\PP N]\cong[M,N\wedge\colim_U S^U]^\PP,
  \end{displaymath}
  where the colimit is over the subrepresentations~$U$ of the universe
  which satisfy~$U^\PP=0$. See~\cite{Mayetal}, XVI.6. In the case at
  hand, we have~$U^\PP=0$ if and only if all the irreducible summands
  of~$U$ are isomorphic to~$D$ or~$H$. This gives an isomorphism
  \begin{displaymath}
    [M,N\wedge\colim_U S^U]^\PP\cong[M,N]^\PP_{\e(D\oplus H)}
  \end{displaymath}
  by taking the colimit out of the brackets.
\end{proof}


\section{Examples:~\texorpdfstring{$Y=S^2\!\times\!S^2$}{Y=S2xS2} and~\texorpdfstring{$Y=K3$}{Y=K3}}

In this section, the~$\PP$-equivariant~$Y$-stable Bauer--Furuta
invariants will be discussed in the two cases~$Y=S^2\!\times\!S^2$
and~$Y=K3$. Computations in our~$\PP$-equivariant stable homotopy
category are much easier than in the general case due to the
simplicity of the universe at hand. The stabilisers occuring here are
only~$1$,~$\TT$, and~$\PP$, with Weyl
groups~\hbox{$W1=\PP$},~\hbox{$W\TT=\ZZ/2$}, and~\hbox{$W\PP=1$},
respectively. Only the latter two are finite. It follows that the
Burnside ring~$[S^0,S^0]^\PP$ of~$\PP$ has rank~2. The following
result is an immediate application of the splitting theorem of tom
Dieck and Segal.

\begin{proposition}\label{prop:free group}
  For~$n\geqslant1$, the group~$[S^0,S^{nD}]^\PP$ is free abelian 
  of rank~$1$, generated by the~$n$-th power of the Euler class~$\e(D)\colon
  S^0\rightarrow S^D$. An isomorphism is given by the mapping degree
  of the~$\PP$-fixed points.
\end{proposition}

Let $\eta\colon S^H\rightarrow S^{3D}$ denote the $\PP$-equivariant
Hopf map. As a non-equivariant map it represents (a suspension of) the
usual Hopf map, and on geometric $\PP$-fixed points it is the identity
of~$S^0$. It follows that~$\Phi^\PP(\eta\e(H))$ is the identity as
well, so that the previous proposition implies the following result.

\begin{corollary}
  The relation
  \begin{equation}
    \label{eq:relation}
    \eta\e(H)=e(D)^3
  \end{equation}
  holds.
\end{corollary}

A similar statement to Proposition~\ref{prop:free group} can be
established for~$H$ instead of~$D$. As it will not be needed in the
following, let us turn towards the examples now.

\begin{proposition}
  The~$\PP$-equivariant Bauer--Furuta invariant of~$S^2\!\times\!S^2$ is
  the Euler class~$\e(D)$ of~$D$. 
\end{proposition}

\begin{proof}
  The index computation shows that the invariant lives
  in~$[S^0,S^D]^\PP$. By Proposition~\ref{prop:free group}, it
  suffices to know map induced on the~$\PP$-fixed points. As already
  remarked in the previous section, this is always the identity.
\end{proof}

\begin{proposition}
  The~$\PP$-equivariant Bauer--Furuta invariant of a~$K3$ surface is
  the~$\PP$-equi\-vari\-ant Hopf map~$\eta$.
\end{proposition}

\begin{proof}
  This time, the index computation shows that the invariant lives
  in the group~$[S^H,S^{3D}]^\PP$. The sphere~$S^H$ sits in a cofibre sequence
  \begin{displaymath}
    S(H)_+\longrightarrow S^0\longrightarrow S^H\longrightarrow\Sigma S(H)_+.
  \end{displaymath}
   It follows that there is an induced exact sequence
  \begin{displaymath}
    [S(H)_+,S^{3D}]^\PP\longleftarrow
    [S^0,S^{3D}]^\PP\longleftarrow
    [S^H,S^{3D}]^\PP\longleftarrow
    [\Sigma S(H)_+,S^{3D}]^\PP.
  \end{displaymath}
  Here, the unit sphere~$S(H)$ in~$H$ is a free,
  2-dimensional~$\PP$-CW-complex with orbit space the real projective
  plane~$\RR P^2$. Therefore, there are
  isomorphisms
  \begin{displaymath}
    [\Sigma^tS(H)_+,S^{3D}]\cong[\RR P^2_+,S^{3-t}],
  \end{displaymath}
  and these groups vanish for~\hbox{$t=0,1$}. As a consequence, the
  middle map in the exact sequence is an
  isomorphism~$[S^H,S^{3D}]^\PP\cong[S^0,S^{3D}]^\PP$, the latter
  group being isomorphic to the integers thanks to
  Proposition~\ref{prop:free group}. Under this isomorphism,
  both~$m(K3)$ and~$\eta$ are sent to~$\e(D)^3$.
\end{proof}

The relation~(\ref{eq:relation}) from the preceding corollary shows
that~$\eta$ is invertible if~$\e(D)$ is invertible, so that there is
an arrow
\begin{displaymath}
  [M,N]^\PP_{\eta}\longrightarrow[M,N]^\PP_{\e(D)}
\end{displaymath}
for all~$\PP$-spectra~$M$ and~$N$, which is the identity on
representatives. Therefore, it sends the~$K3$-stable invariant of a
real spin 4-manifold to its~$S^2\!\times\!S^2$-stable invariant:

\begin{theorem}\label{thm:K3determinesTorus} 
  If~$X$ is a real spin 4-manifold, the~$\PP$-equivariant~$K3$-stable
  Bauer--Furuta invariant~$m(X\#\infty K3)$ of~$X$ determines
  the~\hbox{$S^2\!\times\!S^2$}-stable
  invariant~$m(X\#\infty(S^2\!\times\!S^2))$.
\end{theorem}

This may come as a surprise, in view of the fact that, while
algebraically the intersection form of $S^2\!\times\!S^2$ is an
orthogonal summand of that of~$K3$, geometrically~$S^2\!\times\!S^2$
is not a connected summand of
$K3$. See~\ref{subsec:K3I}.~above. Theorem~\ref{thm:K3determinesTorus}
would follow trivially from the existence of a connected summand
$S^2\!\times\!S^2$ in a connected sum of $K3$ surfaces; but the
hypothetical other summand would be a counterexample to the
$11/8$-conjecture.


\section*{Acknowledgements}

Thanks are due for the helpful comments of Stefan Bauer on an early draft, and
of the anonymous referee on the submitted manuscript.



\vfill

\parbox{\linewidth}{%
Markus Szymik\\
Department of Mathematical Sciences\\
NTNU Norwegian University of Science and Technology\\
7491 Trondheim\\
NORWAY\\
\href{mailto:markus.szymik@ntnu.no}{markus.szymik@ntnu.no}\\
\href{https://folk.ntnu.no/markussz}{folk.ntnu.no/markussz}}

\end{document}